\documentclass[12pt]{article}
\title{A note on one-variable theorems for NSOP}
\author{Will Johnson}

\usepackage{amsmath, amssymb, amsthm}    	
\usepackage{fullpage} 	
\usepackage{amscd}
\usepackage{hyperref}
\usepackage[all]{xy}
\usepackage[T1]{fontenc}
\usepackage{lmodern}
\usepackage{soul}
\usepackage{centernot}
\usepackage[shortlabels]{enumitem}
\usepackage{titlefoot}

\DeclareMathOperator*{\ind}{\raise0.2ex\hbox{\ooalign{\hidewidth$\vert$\hidewidth\cr\raise-0.9ex\hbox{$\smile$}}}}

\newcommand{\hgt}{\operatorname{ht}}

\newcommand{\Aut}{\operatorname{Aut}}

\newcommand{\tp}{\operatorname{tp}}

\newtheorem{theorem}{Theorem}[section] 

\newtheorem{lemma}[theorem]{Lemma}

\newtheorem{fact}[theorem]{Fact}

\newtheorem{question}[theorem]{Question}
\newtheorem{proposition}[theorem]{Proposition}
\newtheorem{proposition-eh}[theorem]{Proposition(?)}
\newtheorem*{theorem-star}{Theorem}
\newtheorem*{conjecture-star}{Conjecture}
\newtheorem*{lemma-star}{Lemma}
\newtheorem{claim}[theorem]{Claim}

\theoremstyle{definition}
\newtheorem{observation}[theorem]{Observation}
\newtheorem{definition}[theorem]{Definition}
\newtheorem{example}[theorem]{Example}

\newtheorem{remark}[theorem]{Remark}

\theoremstyle{remark}

\newtheorem*{acknowledgment}{Acknowledgments}

\newcommand{\Qq}{\mathbb{Q}}

\newcommand{\Zz}{\mathbb{Z}}

\newcommand{\Nn}{\mathbb{N}}

\newcommand{\Mm}{\mathbb{M}}

\newenvironment{claimproof}[1][\proofname]
               {
                 \proof[#1]
                 
               }
               {
                 \endproof
               }

\let\phi\varphi
               
\begin{document}

\maketitle\unmarkedfntext{

}

\begin{abstract}
  We give an example of an SOP theory $T$, such that any
  $L(M)$-formula $\phi(x,y)$ with $|y|=1$ is NSOP.  We show that any
  such $T$ must have the independence property.  We also give a
  simplified proof of Lachlan's \cite{Lachlan} theorem that if every
  $L$-formula $\phi(x,y)$ with $|x|=1$ is NSOP, then $T$ is NSOP.
\end{abstract}

\section{Introduction}
Fix a complete theory $T$.  Recall that a formula $\phi(x,y)$ has the
\emph{strict order property} (SOP) if in some model $M$ there is a
sequence $b_0, b_1,\ldots$ with
\begin{equation*}
  \phi(M,b_0) \subsetneq \phi(M,b_1) \subsetneq \cdots
\end{equation*}
An $L$-theory has the SOP iff some $L$-formula has it.  A formula or
theory is \emph{NSOP} if it doesn't have the SOP.  A classic theorem
of Lachlan \cite{Lachlan} shows that NSOP can be checked on
one-variable formulas:
\begin{theorem}[Lachlan] \label{theorem-0}
  If $T$ has the SOP, then there is an $L$-formula $\phi(x,y)$ with
  the SOP, with $|x|=1$.
\end{theorem}
Here, $|x|$ denotes the length of the tuple of variables $x$.
Lachlan's proof is short but rather convoluted, so in
Section~\ref{simplify} we give what is hopefully a simpler proof.
\begin{remark}
  Lachlan's result is analogous to Shelah's one-variable theorems for
  stability and NIP \cite[Theorems~2.13, 4.6]{Shelah}, Chernikov's
  one-variable theorem for NTP$_2$ \cite[Theorem~2.9, Lemma~3.2]{Ch},
  and Nick Ramsey's one-variable theorem for NSOP$_1$ \cite{Ramsey}.  These
  theorems are identical to Theorem~\ref{theorem-0} with ``SOP''
  replaced with ``unstable'', ``IP'', ``TP$_2$'', and ``SOP$_1$'',
  respectively.
\end{remark}
Lachlan's result suggests the following question:
\begin{question} \label{q1}
  If $T$ has the SOP, is there necessarily an SOP $L$-formula
  $\phi(x,y)$ with $|y|=1$?
\end{question}
The analogous questions for NIP and stability have positive answers
because stability and NIP are symmetric notions: a formula $\phi$ is
stable or NIP iff the opposite formula $\phi^{opp}(y;x) := \phi(x;y)$
is stable or NIP.

In contrast, NSOP is not symmetric, and the answer to
Question~\ref{q1} is \textsc{no}:
\begin{example} \label{between}
  Let $T$ be the theory of the structure $(\Qq,B)$, where $B(x,y,z)$
  means that either $x<y<z$ or $x>y>z$.  Note that $\tp(a,b) =
  \tp(b,a)$ for any distinct singletons $a,b$, in any model of $T$.
  The formula $B(x;y,z)$ has the SOP.  However, no $L$-formula
  $\phi(x;y)$ with $|y|=1$ has the SOP.  Otherwise, take an ascending
  chain
  \begin{equation*}
    \phi(M;b_0) \subsetneq \phi(M;b_1) \subsetneq \cdots
  \end{equation*}
  Then $\tp(b_0,b_1) = \tp(b_1,b_0)$, so $\phi(M;b_1) \subsetneq
  \phi(M;b_0)$, which is absurd.

  The dense circular order is another example, again because $\tp(a,b) =
  \tp(b,a)$ for any two distinct singletons $a,b$.
\end{example}
But what if we allow the formula to mention parameters from the model?
\begin{question} \label{q2}
  If $T$ has the SOP, and $\Mm$ is a monster model of $T$, is there
  necessarily an SOP $L(\Mm)$-formula $\phi(x,y)$ with $|y|=1$?
\end{question}
For instance, in Example~\ref{between}, the formula $\phi(y;z) \equiv
B(0,y;z)$ has the strict order property.  One striking result in this
direction is due to Pierre Simon~\cite{Simon}:
\begin{theorem}[Simon]
  Let $T$ be a theory with monster model $\Mm$.
  \begin{enumerate}
  \item If $T$ is unstable, then some $L(\Mm)$-formula $\phi(x,y)$
    with $|x|=|y|=1$ is unstable.
  \item If $T$ has the independence property, then some
    $L(\Mm)$-formula $\phi(x,y)$ with $|x|=|y|=1$ has the
    independence property.
  \end{enumerate}
\end{theorem}
Could we expect the same to work for NSOP?  Unfortunately, the answer is again \textsc{no}:
\begin{theorem} \label{horror}
  There is a theory $T$ with the SOP, such that every
  $L(M)$-formula $\phi(x,y)$ with $|y|=1$ is NSOP.
\end{theorem}
Although the theory $T$ has a simple description (see
Subsection~\ref{defgoal}), the proof of Theorem~\ref{horror} is
extremely unpleasant, taking up the bulk of this paper.  I would love
to find a simpler counterexample in the spirit of
Example~\ref{between}.

On the other hand, the answer to Question~\ref{q2} is \textsc{yes} if
we assume $T$ is NIP:
\begin{theorem} \label{nipcaseintro}
  Let $T$ be an NIP theory with monster model $\Mm$.  If $T$ has the
  SOP, then there is an $L(\Mm)$-formula $\phi(x,y)$ with the SOP with
  $|x|=|y|=1$.
\end{theorem}
The proof is quite easy, falling directly out of the standard proof
that ``stable $=$ NIP $+$ NSOP''.  We give the details in
Section~\ref{nip-case}.  Theorem~\ref{nipcaseintro} shows that any
theory $T$ as in Theorem~\ref{horror} must have both the SOP and IP,
which perhaps explains why $T$ must be complicated.

Theorem~\ref{horror} can be recast more geometrically as a statement
about definable posets.  Recall that a poset $(P,\le)$ has
\emph{finite height} if there is a finite upper bound $n$ on the
cardinality of chains in $P$.  Any formula $\phi(x,y)$ determines a
partial order on $\Mm^y$ in which
\begin{equation*}
  b < b' \iff \phi(\Mm,b) \subsetneq \phi(\Mm,b').
\end{equation*}
The formula $\phi(x,y)$ has the SOP if and only if this poset has
infinite height.  Theorem~\ref{horror} is thus equivalent to the
following:
\begin{theorem}
  There is an SOP theory $T$, such that any definable poset $(P,\le)$
  with $P \subseteq \Mm^1$ has finite height.
\end{theorem}
This can be contrasted with the situation with linear orders:
\begin{fact}[{\cite[Lemma~5.4]{dpm2}}]
  If there is an infinite definable linear order $(P,\le)$, then there
  is one with $P \subseteq \Mm^1$.
\end{fact}

\subsection{Conventions}
Throughout, letters $a,b,c,\ldots,x,y,z,\ldots$ represent tuples of
elements or variables, rather than single elements or variables.  The
length of a tuple $x$ is written $|x|$.  If $x$ is a tuple of
variables and $M$ is a structure, then $M^x$ is the set of tuples in
$M$ of length $|x|$.  By default we work in a monster model $\Mm$ of a
complete $L$-theory $T$.  If $A$ is a set of parameters, then $L(A)$
is the expansion of $L$ by naming the elements of $A$ as constants.
We try to clarify whether ``formula'' means $L(\Mm)$-formula or
$L$-formula, whenever the difference matters.  A partial order means a
\emph{strict} partial order.  We write disjoint unions as $X \sqcup
Y$.

\subsection{Outline}
Section~\ref{simplify} gives a simplified proof of Lachlan's theorem
mentioned above.  Section~\ref{nip-case} builds off this to prove the
one-variable theorem for NSOP in the NIP case.
Section~\ref{sec4}---the bulk of the paper---works through the theory
of \emph{switchboards}, whose model companion is an example of a
theory as in Theorem~\ref{horror}.
\begin{itemize}
\item Subsection~\ref{s4.1} gives motivation for our choice of theory.
\item Subsection~\ref{defgoal} introduces the theory of switchboards.
\item Subsection~\ref{beginproof} introduces an expansion, the theory
  of \emph{labeled switchboards}.  Labeled switchboards will have the
  amalgamation property, unlike unlabeled switchboards.
\item Subsection~\ref{sefa} verifies most of the amalgamation property
  for labeled switchboards.
\item Subsection~\ref{whoops} builds off this to prove that the theory
  of labeled switchboards has a model completion $T^+$.
\item Subsection~\ref{compare} shows that the original class of
  (unlabeled) switchboards has a model companion $T^-$, definitionally
  equivalent to $T^+$.
\item Subsection~\ref{sop-duh} checks that $T^+$ and $T^-$ have the
  SOP.
\item Finally, Subsection~\ref{endproof} verifies the property that
  $\phi(x,y)$ is NSOP whenever $|y|=1$.
\end{itemize}

\section{A simplified proof of Lachlan's theorem} \label{simplify}
Recall that an $L(\Mm)$-formula $\phi(x,y)$ is NSOP iff the poset of $\phi$-sets
$\{\phi(\Mm,b) : b \in \Mm^y\}$ has finite height.
If $P$ is a poset with finite height, let $\hgt : P \to \Nn$ be the
height function.  Thus $\hgt(a) \ge k$ iff there is a chain $a_0 < a_1
< \cdots < a_k = a$.  If an $L(\Mm)$-formula $\phi(x;y)$ is NSOP and
$b \in \Mm^y$, let $\hgt_\phi(b)$ denote the height of $\phi(\Mm;b)$ in
the poset of $\phi$-sets.  Then
\begin{equation*}
  \phi(\Mm,b) \subsetneq \phi(\Mm,b') \implies \hgt_\phi(b) < \hgt_\phi(b').
\end{equation*}
\begin{lemma} \label{lachlan-step}
  If an $L(\Mm)$-formula $\phi(x,y;z)$ has the SOP, then some
  $L(\Mm)$-formula $\psi(x;z)$ or $\theta(y;z)$ has the SOP.
\end{lemma}
Note that the $x,y,z$ appearing in $\psi(x;z)$ and $\theta(y;z)$ are
the same $x,y,z$ appearing in $\phi(x,y;z)$.  In particular, the
length of $z$ didn't change.
\begin{proof}
  Suppose the lemma fails.  Without loss of generality, $\phi$ is an $L$-formula.
  For $a \in \Mm^x$, let $\phi_a(y;z)$ be $\phi(a,y,z)$.  Then
  $\phi_a$ is NSOP so $\hgt_{\phi_a}(c)$ makes sense for any $a \in
  \Mm^x$ and $c \in \Mm^z$.  For $k \in \Nn$ let $\psi_k(x,z)$ be the
  $L$-formula such that
  \begin{equation*}
    \hgt_{\phi_a}(c) \ge k \iff \Mm \models \psi_k(a,c).
  \end{equation*}
  Take an indiscernible sequence $c_0, c_1, c_2, \ldots$ with
  \begin{equation*}
    \phi(\Mm,c_0) \subsetneq \phi(\Mm,c_1) \subsetneq \cdots
  \end{equation*}
  For any $a \in \Mm^x$ and $k \in \Nn$ we have
  \begin{gather*}
    \phi(a,\Mm,c_0) \subseteq \phi(a,\Mm,c_1) \tag{$\ast$} \\
    \hgt_{\phi_a}(c_0) \le \hgt_{\phi_a}(c_1) \\
    \hgt_{\phi_a}(c_0) \ge k \implies \hgt_{\phi_a}(c_1) \ge k \\
    a \in \psi_k(\Mm,c_0) \implies a \in \psi_k(\Mm,c_1),
  \end{gather*}
  and so $\psi_k(\Mm,c_0) \subseteq \psi_k(\Mm,c_1)$.  Equality must
  hold, or $\psi_k$ has the SOP by indiscernibility.  Then for any $a
  \in \Mm^x$ and $k \in \Nn$, we have
  \begin{gather*}
    a \in \psi_k(\Mm,c_0) \iff a \in \psi_k(\Mm,c_1) \\
    \hgt_{\phi_a}(c_0) \ge k \iff \hgt_{\phi_a}(c_1) \ge k \\    
    \hgt_{\phi_a}(c_0) = \hgt_{\phi_a}(c_1) \\
    \phi(a,\Mm,c_0) = \phi(a,\Mm,c_1),
  \end{gather*}
where the last line follows by ($\ast$).
  Thus $\phi(\Mm,c_0) = \phi(\Mm,c_1)$, a contradiction.
\end{proof}
\begin{theorem}[Lachlan] \label{lachlan-1}
  If an $L(\Mm)$-formula $\phi(x;y)$ has the SOP, then some
  $L(\Mm)$-formula $\psi(z;y)$ has the SOP with $|z|=1$.
\end{theorem}
Note that the length of $y$ is the same in $\phi$ and $\psi$.
\begin{proof}
  By induction on $|x|$ using Lemma~\ref{lachlan-step}.
\end{proof}
\begin{theorem}[Lachlan] \label{lachlan-2}
  If $T$ has the SOP, then some $L$-formula $\phi(x,y)$ with $|x|=1$
  has the SOP.
\end{theorem}
\begin{proof}
  Theorem~\ref{lachlan-1} gives an $L(\Mm)$-formula $\psi(x;y)$ with
  the SOP, with $|x|=1$.  Write $\psi(x,y)$ as $\phi(x,y,c)$ for some
  tuple $c$ in $\Mm$.  Then $\phi(x;y,z)$ has the SOP.  Indeed, if
  \begin{equation*}
    \psi(\Mm,b_0) \subsetneq \psi(\Mm,b_1) \subsetneq \cdots,
  \end{equation*}
  then
  \begin{equation*}
    \phi(\Mm,b_0,c) \subsetneq \phi(\Mm,b_1,c) \subsetneq \cdots \qedhere
  \end{equation*}
\end{proof}

\section{The NIP case} \label{nip-case}
\begin{theorem} \label{nipcase}
  Suppose $T$ is NIP but has the SOP.  Then there is an
  $L(\Mm)$-formula $\phi(x,y)$ with the SOP, with $|x|=|y|=1$.
\end{theorem}
The proof of Theorem~\ref{nipcase} is really just the standard
argument that stability is NSOP plus NIP (see
\cite[Theorem~12.38]{P-book} or \cite[Theorem~2.67]{NIPguide}).
Nevertheless, we trace through the details for the sake of
completeness.
\begin{proof}
  By Theorem~\ref{lachlan-1}, it suffices to find an SOP
  $L(\Mm)$-formula $\phi(x,y)$ with $|y|=1$.  Suppose no such formula
  exists.  Then the following holds:
  \begin{claim} \label{claim1}
    Suppose $C \subseteq \Mm$ is small, $\{b_i\}_{i \in I}$ is a
    $C$-indiscernible sequence of singletons, $q(x)$ is a partial type over $C$, $\phi(x,y)$ is an $L(C)$-formula, and
    $i_0 < i_1$.  Then
    \begin{align*}
      &\{\phi(x,b_{i_0}),\neg \phi(x,b_{i_1})\} \cup q(x) \text{ is consistent} \\
      \text{ if and only if }
      &\{\neg \phi(x,b_{i_0}),\phi(x,b_{i_1})\} \cup q(x) \text{ is consistent}.
    \end{align*}
  \end{claim}
  \begin{claimproof}
    By compactness we can assume $q(x)$ is a single $L(C)$-formula
    $\psi(x)$.  If the claim failed, we would get \begin{align*} \text{either} \qquad &\phi(\Mm,b_{i_0})
    \cap \psi(\Mm) \subsetneq \phi(\Mm,b_{i_1}) \cap \psi(\Mm) \\ \text{ or} \qquad &\phi(\Mm,b_{i_0}) \cap \psi(\Mm) \supsetneq \phi(\Mm,b_{i_1})
    \cap \psi(\Mm).
    \end{align*}
    Either way, the formula $\phi(x,y) \wedge
    \psi(x)$ has the SOP by indiscernibility.
  \end{claimproof}
  Since $T$ is unstable, there is an unstable $L$-formula $\psi(x,y)$
  with $|y|=1$ \cite[Theorem~2.13]{Shelah}.  Take an indiscernible
  sequence $\{(a_i,b_i)\}_{i \in \Qq}$ such that
  \begin{equation*}
    \Mm \models \psi(a_i,b_j) \iff i < j.
  \end{equation*}
  Fix some $n$ and let $[n] = \{1,\ldots,n\}$.  For any $S \subseteq
  [n]$ let $p_S(x)$ be the type
  \begin{equation*}
    p_S(x) = \{\psi(x,b_i) : i \in [n] \cap S\} \cup \{\neg
    \psi(x,b_i) : i \in [n] \setminus S\}
  \end{equation*}
  \begin{claim}
    If $p_S$ is consistent and $\sigma$ is a permutation of $[n]$,
    then $p_{\sigma(S)}$ is consistent.
  \end{claim}
  \begin{claimproof}
    We may assume that $\sigma$ is the permutation $(i ~ i+1)$
    swapping $i$ and $i+1$.  Let
    \begin{equation*}
      q(x) = \{\psi(x,b_j) : j \in [n] \cap S, ~ j \notin \{i,i+1\}\} \cup
      \{\neg \psi(x,b_j) : j \in [n] \setminus S, ~ j \notin \{i,i+1\}\},
    \end{equation*}
    i.e., the part of $p_S$ not involving $i$ and $i+1$.  We must show
    \begin{equation*}
      \{\phi(x,b_i),\neg \phi(x,b_{i+1})\} \cup q(x) \text{ is consistent iff }
      \{\neg \phi(x,b_i),\phi(x,b_{i+1})\} \cup q(x) \text{ is consistent.}
    \end{equation*}
    This follows by Claim~\ref{claim1}, applied to the parameter set
    $C = \{b_j : j \in [n] \setminus \{i,i+1\}\}$ and the
    $C$-indiscernible sequence $\{b_j\}_{i-1 < j < i+2}$.
  \end{claimproof}
  Every $S \subseteq [n]$ has the form $\sigma([i])$ for some $i \le
  n$ and permutation $\sigma : [n] \to [n]$.  Since $p_{[i]}(x)$ is
  realized by $a_{i+0.5}$, it follows that any $p_S$ is consistent.
  As $n$ and $S$ were arbitrary, this contradicts NIP.
\end{proof}

\section{Switchboards} \label{sec4}
\subsection{Motivation} \label{s4.1}
We want to produce a theory $T$ such that
\begin{itemize}
\item some formula $\phi(x,y)$ has
the SOP with $|y| = n > 1$
\item no $L(M)$-formula $\phi(x,y)$ with
$|y|=1$ has the SOP.
\end{itemize}
We may as well take $n=2$.  Then $\phi$ determines a partial order $<$
on $M^2$ with infinite height.  For any fixed $a \in M$, we can
restrict this partial order to $\{a\} \times M$.  The resulting
partial order must have finite height, or some one-variable formula
witnesses SOP.  Then $(\{a\} \times M, <)$ must be a finite union of
antichains.  We may as well require $\{a\} \times M$ to be a single
antichain.  Similarly, we may as well require $M \times \{a\}$ to be
a single antichain.

So we are now considering the theory of structures $(M,<)$ where
$<$ is a partial order on $M^2$, such that each set of the form
$\{a\} \times M$ or $M \times \{a\}$ is an antichain.  The next
natural move is to take the model companion, cross our fingers, and
hope for everything to work out.

This is the strategy we will follow, but with one further
twist: we take the order $<$ on the set of \emph{unordered pairs} rather
than \emph{ordered pairs} to minimize the number of cases that must be
checked in the proof.

\subsection{The definitions and goal} \label{defgoal}
If $M$ is a set, let $[M]^2$ be the set of 2-element subsets of $M$ (not including singletons).
We refer to elements of $[M]^2$ as \emph{edges}, thinking of $M$ as a
complete graph.
\begin{definition}
  A \emph{switchboard} is a structure $(M,<)$, where $<$ is a partial
  order on $[M]^2$, such that $\{x,y\}$ and $\{x,z\}$ are incomparable
  for any distinct $x,y,z \in M$.  In other words, the set of edges
  incident to $x$ is an antichain, for any $x \in M$.  We refer to
  this condition as the \emph{Switchboard Axiom}.
\end{definition}
Officially, we regard switchboards as structures in a language $L^-$
with a 4-ary relation symbol ${<}(x,y,z,w)$ interpreted as
\begin{equation*}
  x \ne y \wedge z \ne w \wedge \{x,y\} < \{z,w\}.
\end{equation*}
In practice, we will \emph{always} regard $<$ as a binary relation on
$[M]^2$, and \emph{never} think of it as a binary relation on $M^2$.
\begin{remark}\label{predet}
  A switchboard structure on $M$ is determined by the information of
  whether $\{x,y\} < \{z,w\}$ holds, for distinct $x,y,z,w \in M$,
  because $\{x,y\} < \{z,w\}$ can only hold when $x,y,z,w$ are all
  distinct.

  For example, there is a unique switchboard structure on any
  3-element set $M$, since we cannot find any four distinct elements
  $x,y,z,w \in M$.
\end{remark}
\begin{theorem}
  \begin{enumerate}
  \item The theory of switchboards has a model companion $T^-$.
  \item $T^-$ is $\aleph_0$-categorical and has the SOP.
  \item If $M \models T^-$ and $\phi(x,y)$ is an $L(M)$-formula with
    $|y|=1$, then $\phi$ is NSOP.
  \end{enumerate}
\end{theorem}
The proof occupies the rest of this paper.

\subsection{Labeled switchboards} \label{beginproof}
Unfortunately, the theory $T^-$ fails to have quantifier elimination,
because the class of switchboards does not have the amalgamation
property (by Remark~\ref{no-amalg} below).  We must first work in an
expanded language, construct its model companion, then relate it back
to the original setting.  In this section, we define the expanded language.
\begin{definition} \label{lsdef}
  A \emph{labeled switchboard} is a structure
  $(M,<,\uparrow,\downarrow)$ where
  \begin{enumerate}
  \item $(M,<)$ is a switchboard.  In particular, $<$ is an order on
    $[M]^2$.
  \item $\uparrow$ and $\downarrow$ are binary relations between $M$
    and $[M]^2$.
  \item (Trichotomy Axiom) \label{trichotomy} For any $a \in M$ and $\{b,c\} \in [M]^2$, exactly one of the
    following holds:
    \begin{itemize}
    \item $a \uparrow \{b,c\}$
    \item $a \in \{b,c\}$
    \item $a \downarrow \{b,c\}$.
    \end{itemize}
\item (Upward Axiom) \label{upward} If $a \uparrow \{b,c\}$ and $\{b,c\} < \{b',c'\}$, then $a
  \uparrow \{b',c'\}$.
\item (Downward Axiom) \label{downward} If $a \downarrow \{b,c\}$ and $\{b,c\} > \{b',c'\}$, then $a
  \downarrow \{b',c'\}$.
  \end{enumerate}
\end{definition}

We pronounce $x \uparrow \{y,z\}$ as ``$x$ favors $\{y,z\}$'' and $x \downarrow \{y,z\}$ as ``$x$ disfavors $\{y,z\}$.''  Officially, we regard labeled switchboards as structures in a language
$L^+$ with a 4-ary relation symbol $<$ and two 3-ary relation symbols
$\uparrow$ and $\downarrow$.  Note that $L^+$ expands $L^-$.  In practice, we will \emph{always} regard $\uparrow, \downarrow$ as binary relations between $M$ and $[M]^2$, and \emph{never} as binary relations between $M$ and $M^2$.
\begin{remark} \label{tripartite}
  The last three axioms of labeled switchboards can be understood as
  saying that for each element $a$, we have a partition of the poset
  $([M]^2,<)$ into three sets:
  \begin{itemize}
  \item An upward-closed set $\{\{x,y\} \in [M]^2 : a \uparrow
    \{x,y\}\}$.
  \item The antichain $\{\{x,y\} \in [M]^2 : a \in \{x,y\}\}$.
  \item A downward-closed set $\{\{x,y\} \in [M]^2 : a \downarrow
    \{x,y\}\}$.
  \end{itemize}
\end{remark}
\begin{remark} \label{determination}
  A labeled switchboard structure with underlying set $M$ is
  determined by the following data:
  \begin{itemize}
  \item How $\{x,y\}$ and $\{z,w\}$ compare, for distinct $x,y,z,w \in
    M$.
  \item Whether $x \uparrow \{y,z\}$ or $x \downarrow \{y,z\}$ holds,
    for distinct $x,y,z \in M$.
  \end{itemize}
  For example, there are eight different labeled switchboard structures on a
  three-element set $\{a,b,c\}$, since the relation $<$ between edges
  can never hold, and then the $\uparrow$ and $\downarrow$ relations
  can be chosen freely.
\end{remark}
\begin{observation} \label{circles}
  Let $P$ be a poset partitioned into three sets $D \sqcup A \sqcup U$, where $A$ is an antichain, $D$ is downward closed, and $U$ is upward closed.  If $a \in A$ and $a < x$, then $x \in U$.  Indeed,
  \begin{itemize}
  \item $x \notin A$ since $A$ is an antichain.
  \item $x \notin D$, or else $a < x$ implies $a \in D$, contradicting
    the fact that $a \in A$.
  \end{itemize}
\end{observation}
\begin{remark} \label{proto}
  In a labeled switchboard,
  \begin{gather*}
    \{a,x\} < \{y,z\} \implies a \uparrow \{y,z\} \\
    \{a,x\} > \{y,z\} \implies a \downarrow \{y,z\}.
  \end{gather*}
  For example, the first line follows by applying
  Observation~\ref{circles} to the partition from
  Remark~\ref{tripartite}: the element $\{a,x\}$ belongs to the
  antichain for $a$, so $\{y,z\}$ must be in the upward-closed set for
  $a$.

  Later, in the model companion, it will turn out that
  \begin{gather*}
    a \uparrow \{y,z\} \iff \exists x : \{a,x\} < \{y,z\} \\
    a \downarrow \{y,z\} \iff \exists x : \{a,x\} > \{y,z\},
  \end{gather*}
  so the two relations $\uparrow$ and $\downarrow$ will be definable
  from $<$.  This will be the bridge between labeled and unlabeled
  switchboards.
\end{remark}
\begin{proposition} \label{convert}
  Every switchboard $(M,<)$ can be expanded to a labeled switchboard
  $(M,<,\uparrow,\downarrow)$.
\end{proposition}
\begin{proof}
  For distinct $a,x,y \in M$,
  \begin{itemize}
  \item let $a \uparrow \{x,y\}$ hold if there is $z \in M$ such that
    $\{a,z\} < \{x,y\}$.
  \item let $a \downarrow \{x,y\}$ hold otherwise.
  \end{itemize}
  Then $(M,<,\uparrow,\downarrow)$ is a labeled switchboard by
  Observation~\ref{trivial} below.
\end{proof}
\begin{observation} \label{trivial}
  Let $(P,<)$ be a poset and $A \subseteq P$ be an antichain.  If we let
  \begin{align*}
    U &= \{x \in P \mid \exists a \in A : x > a\} \\
    D &= P \setminus (U \cup A)
  \end{align*}
  then $U$ is upward-closed, $D$ is downward-closed, and $U \sqcup A
  \sqcup D$ is a partition of $P$.
\end{observation}
\begin{remark} \label{asymmetry}
  The axioms of labeled switchboards are symmetric between $\uparrow$
  and $\downarrow$, but the proof of Proposition~\ref{convert} broke
  this symmetry by treating $\downarrow$ as the default option.  This
  theme will continue in the next section.
\end{remark}

\subsection{Single-element free amalgamation} \label{sefa}
In this section, we show that labeled switchboards can be amalgamated,
which will help construct the model companion in the next section.
Moreover, they can be amalgamated in a specific way (``freely'') which
will be useful in the proof that $\phi(x,y)$ is NSOP whenever $|y|=1$.

First, we reformulate the definition of ``labeled switchboard'' in a
way that is asymmetric and strange, but easier to use for the proof of
free amalgamation.
\begin{lemma} \label{reformulation}
  Let $M$ be a labeled switchboard.  Let $\lhd$ be the binary relation
  on $M \sqcup [M]^2$ defined by $x \lhd y$ if and only if either $x \uparrow
  y$ (so $x \in M$ and $y \in [M]^2$) or $x < y$ (so $x,y \in [M]^2$).  Then $\lhd$ satisfies the following axioms:
  \begin{enumerate}[label=(\arabic*)]
  \item \label{s1} $\lhd$ is transitive.
  \item \label{s2} If $\{x,y\} \lhd \{z,w\}$ then $x \lhd \{z,w\}$.
  \item \label{s3} If $a \lhd b$, then $b \in [M]^2$ (rather than $b
    \in M$).
  \item \label{s4} $\lhd$ is irreflexive.
  \item \label{s6} $x \centernot{\lhd} \{x,y\}$ for $x,y \in M$.
  \end{enumerate}
  Conversely, any relation $\lhd$ satisfying Axioms~\ref{s1}--\ref{s6}
  determines a labeled switchboard structure on $M$.
\end{lemma}
\begin{proof}
  First suppose we have a labeled switchboard.  Axiom \ref{s1} holds
  either because $<$ is a transitive relation on $[M]^2$ or because of
  the Upward Axiom for labeled switchboards: $x \uparrow \{y,z\} <
  \{v,w\} \implies x \uparrow \{v,w\}$.  Axiom \ref{s2} is
  Remark~\ref{proto}.  Axiom \ref{s3} is obvious. Axiom
  \ref{s4} holds because $<$ is irreflexive.  Axiom
  \ref{s6} is part of the Trichotomy Axiom for labeled switchboards.

  Conversely, suppose $\lhd$ is given.  Define
  \begin{align*}
    \{x,y\} < \{z,w\} &\iff \{x,y\} \lhd \{z,w\} \\
    x \uparrow \{y,z\} &\iff x \lhd \{y,z\} \\
    x \downarrow \{y,z\} &\iff (x \notin \{y,z\} \text{ and } x \centernot{\uparrow} \{y,z\}).
  \end{align*}
  Then $<$ is certainly a partial order on $[M]^2$, because $\lhd$ is
  transitive and irreflexive.  If $\{x,y\} < \{x,z\}$ then Axiom
  \ref{s2} gives $x \lhd \{x,z\}$, contradicting Axiom \ref{s6}.  So
  the Switchboard Axiom holds.  Axiom \ref{s6} shows that the two
  cases $x \in\{y,z\}$ and $x \uparrow \{y,z\}$ are mutually
  exclusive, and so the Trichotomy Axiom is automatic by definition of
  $x \downarrow \{y,z\}$.  The Upward Axiom says
  \begin{equation*}
    x \lhd \{y,z\} \lhd \{u,v\} \implies x \lhd \{u,v\},
  \end{equation*}
  which holds by transitivity of $\lhd$.  For the Downward Axiom,
  suppose $x \downarrow \{y,z\} > \{u,v\}$.  We must show $x
  \downarrow \{u,v\}$.  Otherwise, one of two things happens:
  \begin{itemize}
  \item $x \uparrow \{u,v\}$.  Then $x \lhd \{u,v\} \lhd \{y,z\}$ so
    $x \lhd \{y,z\}$.
  \item $x \in \{u,v\}$.  Then $x \in \{u,v\} \lhd \{y,z\}$ so $x \lhd
    \{y,z\}$ by Axiom~\ref{s2}.
  \end{itemize}
  Either way, $x \lhd \{y,z\}$, so $x \uparrow \{y,z\}$, contradicting
  the fact that $x \downarrow \{y,z\}$.  Thus all the axioms of
  labeled switchboards hold.
\end{proof}
\begin{definition}
  A \emph{triangle relation} on $M$ is a relation $\lhd$ on $M \sqcup
  [M]^2$ satisfying the axioms in Lemma~\ref{reformulation}.\footnote{``Triangle'' refers to the shape of the symbol $\lhd$, and has no deeper meaning.}  We call
  these the \emph{triangle axioms}.
\end{definition}
Triangle relations are strict partial orders---they are transitive and
irreflexive.  Our strategy for amalgamating labeled switchboards will
essentially be to freely amalgamate the triangle relations as partial
orders.  For this reason, we first review how strict partial orders
can be amalgamated.
\begin{lemma} \label{poset-amalg}
  Let $X_1, X_2$ be two sets.  Let $<_i$ be a transitive relation on
  $X_i$ for $i = 1, 2$, such that $<_1$ and $<_2$ have the same
  restriction to $X_1 \cap X_2$.  Let $<$ be the transitive closure of
  the union of $<_1$ and $<_2$.
  \begin{enumerate}
  \item $<$ extends $<_i$ on $X_i$.
  \item If $a \in X_1 \setminus X_2$ and $b \in X_2 \setminus X_1$,
    then $a < b$ holds if and only if there is $c \in X_1 \cap X_2$
    such that $a <_1 c <_2 b$.
  \item If $a \in X_2 \setminus X_1$ and $b \in X_1 \setminus X_2$,
    then $a < b$ holds if and only if there is $c \in X_1 \cap X_2$
    such that $a <_2 c <_1 b$.
  \end{enumerate}
\end{lemma}
This is well-known, but we include the proof for completeness.
\begin{proof}
  It suffices to prove the following:
  \begin{claim}
    If $a < b$, then one of the following holds:
    \begin{enumerate}
    \item $a <_i b$ for $i=1$ or $2$.  In particular, $\{a,b\} \subseteq X_i$.
    \item $a \in X_1 \setminus X_2$ and $b \in X_2 \setminus X_1$ and
      $a <_1 c <_2 b$ for some $c \in X_1 \cap X_2$.
    \item $a \in X_2 \setminus X_1$ and $b \in X_1 \setminus X_2$ and
      $a <_2 c <_1 b$ for some $c \in X_1 \cap X_2$.
    \end{enumerate}
  \end{claim}
  Let $<_0$ be the union of $<_1$ and $<_2$.  By construction of the
  transitive closure, there is a sequence
  \begin{equation*}
    a = z_0 <_0 z_1 <_0 \cdots <_0 z_n = b
  \end{equation*}
  with $n > 0$.  Take such a sequence with $n$ minimal.  If $z_j \in
  X_1 \setminus X_2$ for some $0 < j < n$, then we must have $z_{j-1}
  <_1 z_j <_1 z_{j+1}$ (since $z_j$ cannot satisfy $<_2$).  By
  transitivity of $<_1$, we can drop $z_j$ from the list,
  contradicting minimality.  Similarly, $z_j$ cannot be in $X_2
  \setminus X_1$ for $0 < j < n$.  Therefore, $z_j \in X_1 \cap X_2$
  for each $0 < j < n$.  Break into cases depending on $n$.
  \begin{itemize}
  \item $n=1$.  Then $a <_0 b$, so we are in Case 1 of the Claim.
  \item $n=2$.  Then $a <_0 z <_0 b$ for some $z \in X_1 \cap X_2$.
    If $a,b$ are both in $X_1$, then $a <_1 z <_1 b$, contradicting
    minimality.  So one of $a$ and $b$ is in $X_2 \setminus X_1$.
    Similarly, one is in $X_1 \setminus X_2$.  Then we are Case 2 or 3
    of the Claim.
  \item $n>2$.  Take $i$ such that $a \in X_i$.
  Then $a, z_1, z_2$ are in $X_i$, so $a <_i z_1 <_i z_2$.  Then $z_1$
  could be dropped from the list, contradicting minimality.  \qedhere
  \end{itemize}
\end{proof}

\begin{definition} \label{freedef}
  Let $M$ be a labeled switchboard.  Let $S$ be a subset and let $a_1,
  a_2$ be two distinct elements of $M \setminus S$.  Then $a_1$ and $a_2$ are
  \emph{freely amalgamated} over $S$ if the following holds:
  \begin{enumerate}[label=(\roman*)]
  \item \label{b1} If $x,y \in S$, then $\{a_1,x\} < \{a_2,y\}$ if and
    only if there is $\{p,q\} \in [S]^2$ with $\{a_1,x\} <
    \{p,q\}$ and $\{p,q\} < \{a_2,y\}$.
  \item \label{b2} If $x,y \in S$, then $\{a_2,y\} < \{a_1,x\}$ if and
    only if there is $\{p,q\} \in[S]^2$ with $\{a_2,y\} <
    \{p,q\}$ and $\{p,q\} < \{a_1,x\}$.
  \item \label{b3} The edge $\{a_1,a_2\}$ is incomparable to every
    other element of $[S \cup \{a_1,a_2\}]^2$.
  \item \label{b4} If $x \in S$, then $x \downarrow \{a_1,a_2\}$.
  \item \label{b5} If $x \in S$, then $a_1 \uparrow \{a_2,x\}$ if and
    only if there is $\{p,q\} \in [S]^2$ such that $a_1
    \uparrow \{p,q\}$ and $\{p,q\} < \{a_2,x\}$.
  \item \label{b6} If $x \in S$, then $a_2 \uparrow \{a_1,x\}$ if and only if
    there is $\{p,q\} \in [S]^2$ such that $a_2 \uparrow
    \{p,q\}$ and $\{p,q\} < \{a_1,x\}$.
  \end{enumerate}
\end{definition}

\begin{lemma} \label{amalg}
  Suppose $S, A_1, A_2$ are labeled switchboards such that
  \begin{itemize}
  \item $A_1$ and $A_2$ extend $S$.
  \item $A_i = S \cup \{a_i\}$ for $i = 1, 2$, for two distinct
    elements $a_1, a_2 \notin S$.
  \end{itemize}
  Let $M = A_1 \cup A_2 = S \cup \{a_1,a_2\}$.  Then we can make $M$
  into a labeled switchboard extending $A_1$ and $A_2$, in which $a_1$
  and $a_2$ are freely amalgamated over $S$.
\end{lemma}
\begin{proof}
  We may assume $M \cap [M]^2 = \varnothing$, so the disjoint union $M
  \sqcup [M]^2$ is just an ordinary union $M \cup [M]^2$.  For $i=
  1,2$, let $\lhd_i$ be the triangle relation for $A_i$.  Let $\lhd$
  be the transitive closure of the union of $\lhd_1$ and $\lhd_2$.
  Then $\lhd$ is a relation on the set $A_1 \cup [A_1]^2 \cup A_2 \cup
  [A_2]^2$.  We can also regard $\lhd$ as a relation on the bigger set
  $M \cup [M]^2$, which has one new element $\{a_1,a_2\}$.  Thus
  \begin{claim} \label{nodice}
    $\{a_1,a_2\}$ satisfies no instances of $\lhd$.
  \end{claim}
  Note that $(A_1 \cup [A_1]^2) \cap (A_2 \cup [A_2]^2) = S \cup
  [S]^2$.  By Lemma~\ref{poset-amalg}, the following claims hold:
  \begin{claim} \label{internal}
    For $i = 1$ or $2$, $\lhd$ extends the original relation $\lhd_i$
    on $A_i \cup [A_i]^2$.
  \end{claim}
  \begin{claim} \label{bridge}
    If $x$ is in $(A_i \cup [A_i]^2) \setminus (S \cup [S]^2)$,
    and $y$ is in $(A_j \cup [A_j]^2) \setminus (S \cup [S]^2)$ for $i
    \ne j$, then $x \lhd y$ holds if and only if there is $z \in S
    \cup [S]^2$ such that $x \lhd_i z \lhd_j y$.
  \end{claim}
  We first check that $\lhd$ is a triangle relation on $M$, satisfying Axioms~\ref{s1}--\ref{s6}.
  \begin{enumerate}
  \item $\lhd$ is transitive: by construction.
  \item If $\{x,y\} \lhd \{z,w\}$, then $x \lhd \{z,w\}$: By
    construction of the transitive closure, there is some $i=1,2$ and
    edge $p$ such that $\{x,y\} \lhd_i p$ and either $p = \{z,w\}$ or
    $p \lhd \{z,w\}$.  The fact that $\{x,y\} \lhd_i p$ implies $x \lhd_i p$ because $\lhd_i$ itself
    satisfies Axiom \ref{s2} in the triangle axioms.  Since $x \lhd p$ and either $p =
    \{z,w\}$ or $p \lhd \{z,w\}$, it follows that $x \lhd \{z,w\}$ by
    transitivity of $\lhd$.
  \item If $a \lhd b$, then $b \in [M]^2$: By construction of the
    transitive closure there is $b'$ and $i$ such that $b' \lhd_i b$.
    Since $\lhd_i$ itself satisfies Axiom \ref{s3} of the triangle
    axioms, $b$ is in $[M]^2$ rather than $M$.
  \item $\lhd$ is irreflexive: Otherwise $a \lhd a$ for some $a$.
    Then $a \in A_i \cup [A_i]^2$ for some $i$, and $a \lhd_i a$ by
    Claim~\ref{internal}, contradicting the fact that $\lhd_i$ itself is
    irreflexive.
  \item $x \centernot{\lhd} \{x,y\}$ for $x,y \in M$: Suppose $x \lhd
    \{x,y\}$ for the sake of contradiction.
    \begin{itemize}
    \item If $\{x,y\} \subseteq A_i$ for some $i$, then $x \lhd_i
      \{x,y\}$ by Claim~\ref{internal}, contradicting the fact that
      $\lhd_i$ itself satisfies Axiom~\ref{s6} of the triangle axioms.
    \item Otherwise, $\{x,y\} = \{a_1,a_2\}$, and then $x \lhd
      \{a_1,a_2\}$ contradicts Claim~\ref{nodice}.
    \end{itemize}
  \end{enumerate}
  Finally, we check that Conditions~\ref{b1}--\ref{b6} in
  Definition~\ref{freedef} hold.
  \begin{enumerate}[label=(\roman*)]
  \item This says that $\{a_1,x\} \lhd \{a_2,y\}$ if and only if there
    is $p \in [S]^2$ such that $\{a_1,x\} \lhd_1 p \lhd_2 \{a_2,y\}$.
    This is an instance of Claim~\ref{bridge}, since $\{a_1,x\}$ is
    from $[A_1]^2 \setminus [S]^2$ and $\{a_2,y\}$ is from $[A_2]^2
    \setminus [S]^2$.
  \item Similar.
  \item This says that $p \centernot{\lhd} \{a_1,a_2\}$ and
    $\{a_1,a_2\} \centernot{\lhd} p$ for any $p \in [M]^2$.  This
    holds by Claim~\ref{nodice}.
  \item This says that $x \centernot{\lhd} \{a_1,a_2\}$ for any $x \in
    M$.  This holds by Claim~\ref{nodice}.
  \item This says that $a_1 \lhd \{a_2,x\}$ if and only if there is $p \in [S]^2$ such that $a_1 \lhd_1 p \lhd_2 \{a_2,x\}$.  This is another instance of Claim~\ref{bridge}.
  \item Similar.  \qedhere
  \end{enumerate}
\end{proof}
The reader can now safely forget about triangle relations.

\subsection{The model companion $T^+$} \label{whoops}
We can now construct the model companion of the class of labeled switchboards.
\begin{proposition}
  The class of finite labeled switchboards is a Fra\"isse class.
\end{proposition}
\begin{proof}~
  \begin{description}
  \item[Hereditary property:] Clear, since the axioms are
    $\forall$-sentences.
  \item[Amalgamation property:] Suppose we are amalgamating this picture:
    
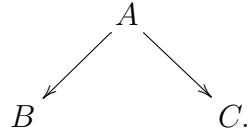
\begin{equation*}
      \xymatrix{& A \ar[dl] \ar[dr] & \\ B & & C.}
    \end{equation*}
    Since we are in a purely relational language, we can use induction
    on the size of $|B \setminus A|$ and $|C \setminus A|$ to reduce
    to the case where $|B \setminus A| = |C \setminus A| = 1$.  This
    case is handled by Lemma~\ref{amalg}.
  \item[Joint embedding property:] Use the amalgamation property over
    the empty labeled switchboard.  \qedhere
  \end{description}
\end{proof}
Let $M_0$ be the Fra\"isse limit of finite labeled switchboards.  Let
$T^+$ be the complete theory of $M_0$.  By general machinery, the
following hold:
\begin{itemize}
\item $T^+$ has quantifier elimination.
\item $T^+$ is complete and countably categorical.
\item $T^+$ is the model companion of labeled switchboards.
\item The class of labeled switchboards has the amalgamation property.
\end{itemize}

\subsection{Comparison to unlabeled switchboards} \label{compare}
Next, we show that $T^+$ is a definitional expansion of some theory $T^-$, which is the model companion of unlabeled switchboards.
\begin{lemma} \label{recover}
  If $x,y,z$ are distinct elements of a model $M$ of $T^+$, then
  \begin{gather*}
    x \uparrow \{y,z\} \iff \exists w \ne x : \{x,w\} < \{y,z\} \\
    x \downarrow \{y,z\} \iff \exists w \ne x : \{x,w\} > \{y,z\}.
  \end{gather*}
\end{lemma}
\begin{proof}
  The right-to-left direction was a consequence of the axioms of
  labeled switchboards; see Remark~\ref{proto}.  For the left-to-right
  direction, suppose $x \uparrow \{y,z\}$.  We need to find $w \in M$
  such that $\{x,w\} < \{y,z\}$.  Since the model $M$ is existentially
  closed among labeled switchboards, it suffices to instead find $w$ in
  an extension $N \supseteq M$.  Let $w$ be a point outside $M$.  Make
  the four-element set $\{x,y,z,w\}$ into a switchboard by making
  $\{x,w\} < \{y,z\}$ and no other relations hold, so that the poset
  $[\{x,y,z,w\}]^2$ looks like this:
  
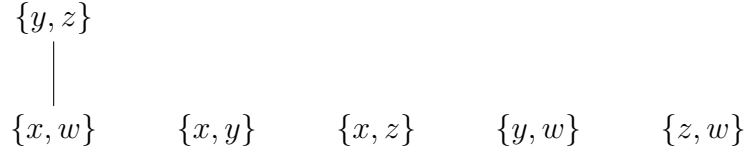
\begin{equation*}
    \xymatrix{ \{y,z\} \ar@{-}[d] & & & & \\
      \{x,w\} & \{x,y\} & \{x,z\} & \{y,w\} & \{z,w\}}
  \end{equation*}
  Expand $\{x,y,z,w\}$ to a labeled switchboard by making
  \begin{itemize}
  \item $x \uparrow \{y,z\}$ and $w \uparrow \{y,z\}$ and $y
    \downarrow \{x,w\}$ and $z \downarrow \{x,w\}$.
  \item $y \uparrow \{x,z\}$ if and only if it holds in $M$.
  \item $z \uparrow \{x,y\}$ if and only if it holds in $M$.
  \item All other information chosen randomly.
  \end{itemize}
  The reader can verify that the Upward and Downward Axioms (i.e.,
  Axioms (\ref{upward}) and (\ref{downward}) in
  Definition~\ref{lsdef}) hold\footnote{Note that these axioms only
  need to be checked relative to the inequality $\{x,w\} < \{y,z\}$,
  as it is the sole inequality that holds.}.

  Use the amalgmation property for labeled switchboards to amalgamate
  $\{x,y,z,w\}$ and $M$ together over $\{x,y,z\}$:
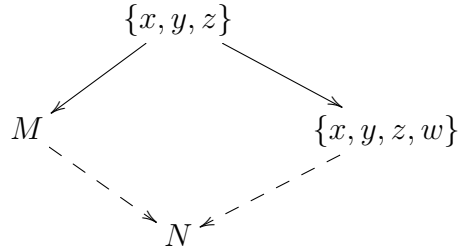
  \begin{equation*}
    \xymatrix{ & \{x,y,z\} \ar[dl] \ar[dr] & \\ M \ar@{-->}[dr] & & \{x,y,z,w\} \ar@{-->}[dl] \\ & N & }
  \end{equation*}
  Then we get a bigger labeled switchboard $N$ extending $M$, and
  containing an element $w$ such that $\{x,w\} < \{y,z\}$.
\end{proof}
So in the theory $T^+$, the two relations $\uparrow$ and $\downarrow$
are definable from the relation $<$.  Let $T^-$ be the reduct of $T^+$
to the language without $\uparrow$ and $\downarrow$.  Then $T^+$ is a
definitional expansion of $T^-$.  So $T^-$ has the same semantic
properties as $T^+$.  For example, $T^-$ is countably categorical and
complete.
\begin{theorem}
  $T^-$ is the model companion of (unlabeled) switchboards.
\end{theorem}
\begin{proof}
  It suffices to check the following three claims:
  \begin{enumerate}
  \item \emph{Every switchboard embeds into a model of $T^-$:} first use
    Proposition~\ref{convert} to expand $M$ to a labeled switchboard,
    then embed it into a model of $T^+$, then take the reduct to the
    language $L^-$.
  \item \emph{Every model of $T^-$ is a switchboard:} clear.
  \item \emph{$T^-$ is model complete}: this holds because $T^-$ has
    quantifier elimination after adding the two symbols $\uparrow,
    \downarrow$, and these symbols are both existentially definable
    and universally definable.  The existentialy definability is
    Lemma~\ref{recover}.  The universal definability holds because
    $\uparrow$ and $\downarrow$ are essentially each other's
    complements, by the Trichotomy Axiom. \qedhere
  \end{enumerate}
\end{proof}
\begin{remark}\label{no-amalg}
  The theory $T^-$ does not have quantifier elimination.  Otherwise,
  every switchboard $M$ would have a unique extension to a labeled
  switchboard, contradicting the fact that the unlabeled switchboard
  with three elements has eight extensions to a labeled switchboard
  (Remarks~\ref{predet} and \ref{determination}).  It follows that the
  class of switchboards does not have the amalgamation property.
  Concrete failures of the AP can be extracted from the proof of
  Lemma~\ref{recover}.
\end{remark}

\subsection{The theory has the SOP} \label{sop-duh}
Make $\Zz$ into a switchboard by ordering
\begin{equation*}
  \cdots < \{0,1\} < \{2,3\} < \{4,5\} < \cdots
\end{equation*}
and making everything else in $[\Zz]^2$ be incomparable.  Embed $\Zz$
into a model $(M,<)$ of $T^-$.  Then the poset $([M]^2,<)$ has an
infinite chain, inherited from $([\Zz]^2,<)$.  Therefore $T^-$ and
the definitionally equivalent theory $T^+$ (Section~\ref{whoops}) have the SOP.

\subsection{One-variable formulas} \label{endproof}
It remains to verify that $T^{\pm}$ is NSOP for formulas $\phi(x,y)$
with $|y|=1$.  Broadly speaking, the strategy is as follows.  If such
a formula $\phi$ existed, there would be a definable strict partial
order $(D,<)$ with $D \subseteq M^1$, and an ascending indiscernible
sequence $a_0 < a_1 < a_2 < \cdots$ in $(D,<)$.  We build a new
sequence $c_0, c_1, c_2, \ldots$ in such a way that $c_ic_{i+1} \equiv
a_0a_1$ for each $i$, but the $c_i$ are otherwise ``indepedent'' in
some sense related to our earlier free amalgamation.  By carefully
analyzing different levels of independence, we will show that $c_0c_n
\equiv c_nc_0$ for sufficiently large $n$.  But $c_0 < c_1 < \cdots <
c_n$, so it is impossible for $\tp(c_0,c_n)$ to equal $\tp(c_n,c_0)$.

For the remainder of the paper, work in a monster model $\Mm$ of $T^+$.  Definitions~\ref{halfshard} and \ref{disting} below should be seen as variants of ``independence'', in some very loose sense.
\begin{definition} \label{halfshard}
  Let $B \subseteq \Mm$ be small, and let $a_1, a_2 \in \Mm \setminus B$ be two
  singletons with $a_1 \equiv_B a_2$.
  \begin{enumerate}
  \item $\tp(a_1,a_2/\Mm)$ is \emph{half-symmetric} if
    \begin{gather*}
      \{a_1,b\} < \{a_2,c\} \iff \{a_2,b\} < \{a_1,c\} \\
    \end{gather*}
    for any $\{b,c\} \in [B]^2$.
  \item $\tp(a_1,a_2/\Mm)$ is \emph{symmetric} if it is
    half-symmetric, and
    \begin{gather*}
      a_1 \uparrow \{a_2,b\} \iff a_2 \uparrow \{a_1,b\} \\
      a_1 \downarrow \{a_2,b\} \iff a_2 \downarrow \{a_1,b\}
    \end{gather*}
    for any $b \in B$.
  \end{enumerate}
\end{definition}
\begin{remark}
  If $a_1, a_2 \in \Mm \setminus B$ are singletons and $a_1 \equiv_B a_2$, then
  $\tp(a_1,a_2/B)$ is symmetric if and only if $a_1a_2 \equiv_B
  a_2a_1$.  This can be seen by quantifier elimination---the only
  atomic formulas in $\tp(a_1,a_2/B)$ beyond those in $\tp(a_1/B) \cup
  \tp(a_2/B)$ are the formulas
  \begin{gather*}
    \{x_1,b\} < \{x_2,c\}, \text{ etc.} \\
    x_1 \uparrow \{x_2,b\}, \text{ etc.}
  \end{gather*}
  appearing in Definition~\ref{halfshard}, and intrinsically symmetric formulas like
  \begin{gather*}
    \{x_1,x_2\} < \{b,c\}, \text{ etc.} \\
    b \uparrow \{x_1,x_2\}, \text{ etc.}
  \end{gather*}
\end{remark}
\begin{definition} \label{disting}
  Let $a_1,a_2$ be distinct singletons in $\Mm \setminus B$.  Then
  $\tp(a_1,a_2/B)$ is \emph{distinguished} if the following two
  conditions both hold:
  \begin{itemize}
  \item For any $b, c \in B$, if $\{a_1,b\} > \{a_2,c\}$, then there
    is $\{u,v\} \in [B]^2$ such that $\{a_1,b\} > \{u,v\} >
    \{a_2,c\}$.
  \item For any $b, c \in B$, if $\{a_1,b\} < \{a_2,c\}$, then there
    is $\{u,v\} \in [B]^2$ such that $\{a_1,b\} < \{u,v\} <
    \{a_2,c\}$.
  \end{itemize}
\end{definition}
\begin{lemma}
  If $a_1 \equiv_B a_2$ and $\tp(a_1,a_2/B)$ is distinguished, then
  $\tp(a_1,a_2/B)$ is half-symmetric.
\end{lemma}
\begin{proof}
  We must show that for any $b, c \in B$,
  \begin{equation*}
    \{a_1,b\} < \{a_2,c\} \iff \{a_2,b\} < \{a_1,c\}.
  \end{equation*}
  We prove the $\Rightarrow$ direction; the $\Leftarrow$ direction is
  similar.  Suppose $\{a_1,b\} < \{a_2,c\}$.  By definition of
  ``distinguished'', there is $\{u,v\} \in [B]^2$ such that $\{a_1,b\}
  < \{u,v\} < \{a_2,c\}$.  Since $a_1 \equiv_B a_2$, and $\{b,c,u,v\}
  \subseteq B$, we have
  \begin{gather*}
    \{a_1,b\} < \{u,v\} \implies \{a_2,b\} < \{u,v\} \\
    \{u,v\} < \{a_2,c\} \implies \{u,v\} < \{a_1,c\}.
  \end{gather*}
  Then
  \begin{equation*}
    \{a_2,b\} < \{u,v\} < \{a_1,c\},
  \end{equation*}
  so $\{a_2,b\} < \{a_1,c\}$ as desired.
\end{proof}
\begin{lemma} \label{yuck}
  Let $B$ be a small set and let $c_1, c_2, c_3$ be distinct elements of $\Mm \setminus B$, all realizing the same type over $B$.  Suppose
  \begin{itemize}
  \item $c_1$ and $c_3$ are freely amalgated over $Bc_2$ in the sense
    of Definition~\ref{freedef}.
  \item $\tp(c_1,c_2/B)$ and $\tp(c_1,c_3/B)$ and $\tp(c_2,c_3/B)$ are
    distinguished, hence half-symmetric.
  \end{itemize}
  Then $\tp(c_1,c_3/B)$ is symmetric.
\end{lemma}
\begin{proof}
  By assumption, $\tp(c_1,c_3/B)$ is half-symmetric, so it remains to
  prove that
  \begin{equation*}
    c_1 \uparrow \{c_3,b\} \iff c_3 \uparrow \{c_1,b\},
  \end{equation*}
  for $b \in B$.  By symmetry, it suffices to prove the $\Rightarrow$
  direction.  Suppose $c_1 \uparrow \{c_3,b\}$.  By the definition of
  free amalgamation (specifically, Condition~\ref{b5} in Definition~\ref{freedef}), there is $\{u_0,v_0\}
  \subseteq Bc_2$ such that
  \begin{equation*}
    c_1 \uparrow \{u_0,v_0\} \text{ and } \{u_0,v_0\} < \{c_3,b\}.
  \end{equation*}
  \begin{claim}
    There is $\{u,v\} \subseteq B$ such that
    \begin{equation*}
      c_1 \uparrow \{u,v\} \text{ and } \{u,v\} < \{c_3,b\}.
    \end{equation*}
  \end{claim}
  \begin{claimproof}
    If $\{u_0,v_0\} \subseteq B$, take $\{u,v\} = \{u_0,v_0\}$.
    Otherwise, $c_2 \in \{u_0,v_0\}$, so $\{u_0,v_0\} = \{c_2,b'\}$
    for some $b' \in B$.  Then
    \begin{equation*}
      c_1 \uparrow \{c_2,b'\} \text{ and } \{c_2,b'\} < \{c_3,b\}.
    \end{equation*}
    Since $\tp(c_2,c_3/B)$ is distinguished, there is some $\{u,v\}
    \in [B]^2$ such that
    \begin{equation*}
      \{c_2,b'\} < \{u,v\} < \{c_3,b\}.
    \end{equation*}
    By the Upward Axiom of labeled switchboards,
    \begin{equation*}
      c_1 \uparrow \{c_2,b'\} < \{u,v\} \implies c_1 \uparrow \{u,v\}. \qedhere
    \end{equation*}
  \end{claimproof}
  Since $c_1 \equiv_B c_3$ and $\{u,v,b\} \subseteq B$,
  \begin{gather*}
    c_1 \uparrow \{u,v\} \implies c_3 \uparrow \{u,v\} \\
    \{u,v\} < \{c_3,b\} \implies \{u,v\} < \{c_1,b\}.
  \end{gather*}
  Finally,
  \begin{equation*}
    c_3 \uparrow \{u,v\} < \{c_1,b\} \implies c_3 \uparrow \{c_1,b\}
  \end{equation*}
  by the Upward Axiom.
\end{proof}
\begin{lemma} \label{amalgamator}
  Let $B$ be a small set and let $a_1, a_2$ be two elements of $\Mm
  \setminus B$.  Then there is $\sigma \in \Aut(\Mm/B)$ such that
  $a_1$ and $\sigma(a_2)$ are freely amalgamated over $B$.
\end{lemma}
\begin{proof}
  This follows formally from Lemma~\ref{amalg} and the fact that $T^+$
  is the model completion of labeled switchboards.  More precisely,
  take two distinct elements $c_1, c_2$ outside $B$, and make $B \cup
  \{c_i\}$ into a labeled switchboard isomorphic to $B \cup \{a_i\}$.
  Use Lemma~\ref{amalg} to make $B \cup \{c_1,c_2\}$ into a labeled
  switchboard in which $c_1$ and $c_2$ are freely amalgamated.  Use
  quantifier elimination and the fact that $\Mm$ is a monster model to
  embed $B \cup \{c_1,c_2\}$ into $\Mm$ over $B$.  The images of
  $c_1$ and $c_2$ give two elements $e_1, e_2 \in \Mm$ such that $e_i
  \equiv_B c_i \equiv_B a_i$ for $i =1,2$, and $e_1$ and $e_2$ are
  freely amalgated over $B$.  Use a further automorphism to move $e_1$
  to $a_1$.
\end{proof}
The next proposition is the technical core of the proof.
\begin{proposition} \label{core}
  Let $B$ be a finite subset of $\Mm$, let $p(x)$ be a complete 1-type
  over $B$ not realized in $B$, and let $q(x,y)$ be a complete 2-type over $B$ extending
  $p(x) \cup p(y) \cup \{x \ne y\}$.
  \begin{enumerate}
  \item There exists a sequence $c_0, c_1, c_2, \ldots$ of
    realizations of $p$, such that
    \begin{itemize}
    \item $c_ic_{i+1}$ realizes $q$ for each $i$.
    \item For $i \ge 2$, $c_i$ and $c_0$ are freely amalgamated
      over $Bc_{i-1}$.
    \end{itemize}
  \end{enumerate}
  For the remaining two points fix a sequence $c_0, c_1, c_2, \ldots$
  as in the the previous point, and assume that the elements $c_0,
  c_1, c_2, \ldots$ are pairwise distinct.
  \begin{enumerate}[resume]
  \item If $i \ge |B|$, then $\tp(c_0,c_i/B)$ is distinguished, hence
    half-symmetric.
  \item If $i > |B|$, and $q$ is distinguished, then $\tp(c_0,c_i/B)$
    is symmetric.
  \end{enumerate}
\end{proposition}
\begin{proof}
  \begin{enumerate}
  \item Take $c_0c_1$ to be any realization of $q$.  For $i \ge 2$,
    take $c_i$ such that $c_{i-1}c_i \models q$, then use
    Lemma~\ref{amalgamator} to move $c_i$ by an automorphism over
    $Bc_{i-1}$ to make $c_i$ and $c_0$ be freely amalgated over
    $Bc_{i-1}$.
  \item First we prove the
    following claim:
    \begin{claim}
      Suppose $i > 0$ and $b_i, b_0$ are elements of $B$ such that
      $\{c_i,b_i\} > \{c_0,b_0\}$.  Then one of the following happens:
      \begin{enumerate}
      \item There is $\{u,v\} \in [B]^2$ such that
        $\{c_i,b_i\} > \{u,v\} > \{c_0,b_0\}$.
      \item There are $b_1,\ldots,b_{i-1} \in B$ such that
        \begin{equation*}
          \{c_i,b_i\} > \{c_{i-1},b_{i-1}\} > \cdots > \{c_0,b_0\}.
        \end{equation*}
      \end{enumerate}
    \end{claim}
    \begin{claimproof}
      Proceed by induction on $i$.  When $i = 1$, Case~(b)
      holds trivially.  Suppose $i > 1$.  Then $c_i$ and $c_0$ are
      freely amalgamated over $Bc_{i-1}$.  Because of the definition
      of free amalgamation (i.e., Conditions~\ref{b1}, \ref{b2} of Definition~\ref{freedef}), there must be
      $\{u,v\} \in [Bc_{i-1}]^2$ such that $\{c_i,b_i\} > \{u,v\} >
      \{c_0,b_0\}$.  If $\{u,v\} \subseteq B$ then we are in
      Case~(a).  Otherwise, $c_{i-1} \in \{u,v\}$, so
      $\{u,v\} = \{c_{i-1},b_{i-1}\}$ for some $b_{i-1} \in B$.  Then
      $\{c_{i-1},b_{i-1}\} > \{c_0,b_0\}$, so by induction, one of the
      following holds:
      \begin{itemize}
      \item There are $\{u,v\} \subseteq B$ such that
        $\{c_{i-1},b_{i-1}\} > \{u,v\} > \{c_0,b_0\}$.  Then
        \begin{equation*}
          \{c_i,b_i\} > \{c_{i-1},b_{i-1}\} > \{u,v\} > \{c_0,b_0\}
        \end{equation*}
        and we are in Case~(a).
      \item There are $b_{i-2},\ldots,b_1$ such that
        \begin{equation*}
          \{c_i,b_i\} > \{c_{i-1},b_{i-1}\} > \{c_{i-2},b_{i-2}\} >
          \cdots > \{c_0,b_0\}.
        \end{equation*}
        Then we are in Case~(b).  \qedhere
      \end{itemize}
    \end{claimproof}
    Now suppose $i \ge |B|$.  We must show that $\tp(c_0,c_i/B)$ is distinguished.  There are two points to check in Definition~\ref{disting}.  We check the second point; the first is similar\footnote{\ldots using a variant of the claim where we
    replace $>$ with $<$.  The proof is identical, but we can't just
    say it works ``by symmetry'' since the definition of free
    amalgamation broke the symmetry between up and down in the
    poset.}  Suppose $b_0, b_i \in B$ are such that
    \begin{equation*}
      \{c_i,b_i\} > \{c_0,b_0\}.
    \end{equation*}
    We must find $\{u,v\} \in [B]^2$ such that
    \begin{equation*}
      \{c_i,b_i\} > \{u,v\} > \{c_0,b_0\}.
    \end{equation*}
    Otherwise, the Claim gives $b_1, b_2, \ldots, b_{i-1} \in B$ such that
    \begin{equation*}
      \{c_i,b_i\} > \{c_{i-1},b_{i-1}\} > \cdots > \{c_0,b_0\}.
    \end{equation*}
    By the pigeonhole principle, there are $j_1 < j_2 \le i$ with
    $b_{j_1} = b_{j_2} =: b$.  Then $\{c_{j_1},b\} > \{c_{j_2},b\}$,
    which contradicts the Switchboard Axiom.
  \item By the previous point and the assumption, $\tp(c_0,c_{i-1}/B)$ and $\tp(c_{i-1}c_i/B) = q$ are both distinguished.  By the previous point, $\tp(c_0,c_i/B)$ is distinguished.  By Lemma~\ref{yuck}, applied to the three elements $c_0, c_{i-1}, c_i$, we see that $\tp(c_0,c_i/B)$ is symmetric.  \qedhere
  \end{enumerate}
\end{proof}
\begin{theorem}
  Suppose $M$ is a model of $T^+$ and $\phi(x,y)$ is an $L(M)$-formula
  with $|y|=1$.  Then $\phi(x,y)$ is NSOP.
\end{theorem}
\begin{proof}
  We may assume $M$ is a monster model $\Mm$.  Take a finite set $B$
  containing all the parameters in $\phi(x,y)$, so that $\phi$ is an
  $L(B)$-formula.  Take a $B$-indiscernible sequence $a_0, a_1, a_2,
  \ldots$ in $\Mm^1$ such that
  \begin{equation*}
    \phi(\Mm,a_0) \subsetneq \phi(\Mm,a_1) \subsetneq \cdots
  \end{equation*}
  Let $p = \tp(a_i/B)$ for any $i$.  Let $q = \tp(a_i,a_j/B)$ for any
  $i < j$.  If $(b,c) \models q$, then $b, c \models p$, and
  $\phi(\Mm,b) \subsetneq \phi(\Mm,c)$.

  Let $c_0, c_1, c_2, \ldots$ be a sequence as in
  Proposition~\ref{core}, with respect to our chosen $p$ and $q$.  In
  particular, $c_ic_{i+1}$ realizes $q$ for each $i$, so
  \begin{equation*}
    \phi(\Mm,c_0) \subsetneq \phi(\Mm,c_1) \subsetneq \cdots
  \end{equation*}
  which implies that the $c_i$ are pairwise distinct.  Take $n = |B|$.
  By part (2) of Proposition~\ref{core}, $q' = \tp(c_0,c_n/B)$ is
  distinguished.  Note that if $(b,c) \models q'$, then $\phi(\Mm,b)
  \subsetneq \phi(\Mm,c)$, because $\phi(\Mm,c_0) \subsetneq
  \phi(\Mm,c_n)$.

  Let $c'_0, c'_1, c'_2, \ldots$ be a sequence as in
  Proposition~\ref{core}, with respect to $p$ and $q'$.  In
  particular, $c'_ic'_{i+1}$ realizes $q'$ for each $i$, so
  \begin{equation*}
    \phi(\Mm,c'_0) \subsetneq \phi(\Mm,c'_1) \subsetneq \cdots
  \end{equation*}
  which implies that the $c'_i$ are pairwise distinct.  Take $n =
  |B|+1$.  By part (3) of Proposition~\ref{core}, $\tp(c'_0,c'_n/B)$ is
  symmetric.  Then
  \begin{equation*}
    \phi(\Mm,c'_0) \subsetneq \phi(\Mm,c'_n) \implies \phi(\Mm,c'_n) \subsetneq \phi(\Mm,c'_0),
  \end{equation*}
  a contradiction.
\end{proof}

\begin{acknowledgment}
  The author was supported by the Ministry of Education of
  China (Grant No.\@ 22JJD110002).
  Nick Ramsey encouraged the author to write up this note.  The author would also like to thank the anonymous referee, who provided many helpful comments.
\end{acknowledgment}

\bibliographystyle{plain} \bibliography{mybib}{}

\begin{thebibliography}{1}

\bibitem{Ch}
Artem Chernikov.
\newblock Theories without the tree property of the second kind.
\newblock {\em Annals of Pure and Applied Logic}, 165(2):695--723, February
  2014.

\bibitem{dpm2}
Will Johnson.
\newblock The classification of dp-minimal and dp-small fields.
\newblock {\em J. Eur. Math. Soc.}, 25(2):467--513, July 2023.

\bibitem{Lachlan}
Alistair~H. Lachlan.
\newblock A remark on the strict order property.
\newblock {\em Mathematical Logic Quarterly}, 21(1):69--70, 1975.

\bibitem{P-book}
Bruno Poizat.
\newblock {\em A course in model theory}.
\newblock Springer-Verlag, 2000.

\bibitem{Ramsey}
Nicholas Ramsey.
\newblock A note on {NSOP$_1$} in one variable.
\newblock {\em J. Symbolic Logic}, 84(1):388--392, March 2019.

\bibitem{Shelah}
Saharon Shelah.
\newblock Stability, the f.c.p., and superstability; model theoretic properties
  of formulas in first order theory.
\newblock {\em Annals of Mathematical Logic}, 3(3):271--362, October 1971.

\bibitem{NIPguide}
Pierre Simon.
\newblock {\em A guide to NIP theories}.
\newblock Lecture Notes in Logic. Cambridge University Press, July 2015.

\bibitem{Simon}
Pierre Simon.
\newblock A note on {NIP} and stability in one variable.
\newblock {\tt arXiv:2103.15799v1 [math.LO]}, March 2021.

\end{thebibliography}

\end{document}